\documentclass[12pt,a4paper]{article}
\usepackage{amsmath,amssymb,amsthm,amsfonts}
\usepackage{centernot}
\usepackage{mathtools}
\usepackage{stmaryrd}
\usepackage{color}
\usepackage{enumitem}
\usepackage[english]{babel}
\usepackage{url}
\usepackage{graphicx}
\usepackage{a4}
\newtheorem{theorem}{Theorem}

\newtheorem{prop}[theorem]{Proposition}
\newtheorem{lemma}[theorem]{Lemma}
\newtheorem{cor}[theorem]{Corollary}
\newtheorem{remark}[theorem]{Remark}

\usepackage{multirow}

\begin{document}
	\baselineskip 18pt

	\title{\bf Units and Augmentation Powers in Integral Group Rings}

	\author{Sugandha Maheshwary{\footnote{Research supported by DST, India (INSPIRE/04/2017/000897) is gratefully acknowledged.} 
		}
		\\ {\em \small Indian Institute of Science Education and Research, Mohali,}\\
		{\em \small Sector 81, Mohali (Punjab)-140306, India.}
		\\{\em \small email: sugandha@iisermohali.ac.in}
		\and Inder Bir S. Passi {\footnote{Corresponding author}} \\  {\em \small Centre for Advanced Study
			in Mathematics,}\\ {\em \small Panjab University, Chandigarh-160014, India} \\
		{\small \& }\\
		{\em \small Indian Institute of Science Education and Research, Mohali,}\\
		{\em \small Sector 81, Mohali (Punjab)-140306, India.}\\{\em \small email: ibspassi@yahoo.co.in } }
	\date{}
	{\maketitle}
	
	\begin{abstract}\noindent The augmentation powers in an integral group ring $\mathbb ZG$ induce a natural filtration of the unit group of $\mathbb ZG$ analogous to the filtration of the group $G$ given by its dimension series $\{D_n(G)\}_{n\ge 1}$. The purpose of the present  article is to investigate this filtration, in particular, the triviality of its intersection.
		\end{abstract}\vspace{.25cm}
	 \vspace{.25cm} 

	\section{Introduction}
	
	Given a group $G$, the powers $\Delta^n(G),\ n\geq 1$, of the augmentation ideal $\Delta(G)$ of its integral group ring $\mathbb ZG$ induce a $\Delta$-adic filtration of $G$, namely, the one given by its dimension subgroups defined by setting $$D_n(G)=G\cap (1+\Delta^n(G)),\ n=1,2,3,\ldots$$ The dimension series of $G$ has been a subject of intensive research (see \cite{Pas79}, \cite{GuP87}, \cite{MP09}). This filtration of a group $G$ (group of trivial units in $\mathbb{Z}G$),  suggests its natural extension to the full unit group $\mathcal V(\mathbb ZG)$ of normalized units, i.e., the group of units of augmentation one in $\mathbb ZG$, by setting $$\mathcal V_n(\mathbb ZG)=\mathcal V(\mathbb ZG)\cap(1+\Delta^n(G)),\ n=1,2,3\ldots$$ It is easy to see that $\{\mathcal V_n(\mathbb ZG)\}_{n\ge 1}$ is a central series in $\mathcal V(\mathbb ZG)$ and consequently, for every $n\ge 1$, the $n^\mathrm{th}$ term $\gamma_n(\mathcal V(\mathbb ZG))$ of the lower central series $\{\gamma_n(\mathcal V(\mathbb ZG))\}_{n=1}^\infty$ of $\mathcal V(\mathbb ZG)$ is contained in $\mathcal V_n(\mathbb ZG)$. 
	Thus the triviality of the  {\it $\Delta$-adic residue} of $\mathcal{V}(\mathbb{Z}G)$ $$\mathcal{V}_\omega(\mathbb{Z}G):=\cap_{n=1}^\infty\mathcal{V}_n(\mathbb{Z}G)$$ implies the residual nilpotence of $\mathcal{V}(\mathbb ZG)$, and so, in particular,  motivates its further  investigation. 
	\par\vspace{0.4cm}
The paper is structured as follows: We begin by collecting some results about the above stated filtration in Section 2. In Section 3, we take up the investigation of groups $G$ with $\mathcal{V}_\omega(\mathbb{Z}G)=\{1\}$. Section 4 deals with the groups having trivial $\mathcal D$-residue, i.e., with groups whose dimension series intersect in identity.
	\par\vspace{0.4cm}	
 We now  mention  some of our main results. 
	\par\vspace{0.4cm}
If $G$ is a finite group, then the filtration $\{\mathcal{V}_n(\mathbb{Z}G)\}_{n\geq 1}$ terminates, i.e.,\linebreak $\mathcal{V}_n(\mathbb{Z}G)=\{1\}$, for some $n\geq 1$, if and only if \\
$(i)$ $G$ is an abelian group of exponent $2,\,3,\,4$ or $6$; or\\
$(ii)$ $G=K_8\times E$, where $K_8$ denotes the quaternion group of order $8$ and $E$ denotes an elementary abelian 2-group (Theorem \ref{theo:Finite_Terrminates}). 
Further,  $\mathcal{V}(\mathbb{Z}G)$ has trivial $\Delta$-adic residue if, and only if, \\
$(i)$ $G$ is an abelian group of exponent $6$; or\\
$(ii)$ $G$ is a $p$-group
 (Theorem \ref{Finite_TrivialResidue}).\\
 It is  interesting to compare the constraints obtained on finite groups satisfying the specified conditions.

For an arbitrary group $G$, if $\mathcal{V}(\mathbb{Z}G)$ has trivial $\Delta$-adic residue, then $G$ cannot have an element of order $pq$ with primes $p<q$, except possibly when $(p,\,q)= (2,\,3)$ (Theorem \ref{2,3_only}). Furthermore, if $G$ is a nilpotent group which is 
			$\{2,\,3\}$-torsion-free, then, for $\Delta$-adic residue of $\mathcal{V}(\mathbb{Z}G)$ to be trivial, the torsion subgroup $T$ of $G$ must satisfy one of the following conditions:\\		
		$(i)$  $T=\{1\}$ i.e., $G$ is a torsion-free nilpotent group, or;\\
		$(ii)$ $T$ is a $p$-group with no element of infinite $p$-height (Theorem \ref{nil}).\\		

In addition, if $G$ is abelian, then it turns out that $\mathcal{V}(\mathbb{Z}G)$ possesses trivial \linebreak$\Delta$-adic residue, if and only if $\mathcal{V}(\mathbb{Z}T)$ does (Theorem \ref{abelian}). 
 
 	We briefly examine the class 
 	 $\mathcal C$ of groups $G$ with $\mathcal{V}_\omega(\mathbb{Z}G)=\{1\}$, and prove that a group $G$ belongs to $\mathcal C$ if  all its quotients $G/\gamma_n(G)$ do so (Theorem \ref{nilpotentquotient}), and that this class is closed under discrimination (Theorem \ref{disc}). Finally, we \linebreak examine the groups $G$ which have the property that the dimension series $\{D_{n,\,\mathbb Q}(G)\}_{n\ge 1}$ over the rationals has non-trivial intersection while 
 	 $\{D_{n}(G)\}_{n\ge 1}$, the one over\linebreak integers,  has trivial intersection (Theorem \ref{dimension}).  
\par\vspace{.5cm}
 For basic results on units and augmentation powers in group rings, we refer the reader to \cite{Seh78} and \cite{Pas79}.
 
\section{The filtration $\mathcal{V}_n(\mathbb{Z}G)$}
As mentioned in the introduction, the filtration $\{\mathcal{V}_n(\mathbb{Z}G)\}_{n=1}^{\infty}$ of the group $\mathcal{V}(\mathbb{Z}G)$  of the normalized units is given by $$\mathcal V_n(\mathbb ZG)=\mathcal V(\mathbb ZG)\cap(1+\Delta^n(G)),\ n=1,2,3\ldots$$ and for every $n\geq 1$, 
\begin{equation}\nonumber
\gamma_n(\mathcal V(\mathbb ZG))\subseteq \mathcal V_n(\mathbb ZG),
\end{equation}
where $\gamma_n(\mathcal V(\mathbb ZG))$ denotes the $n^{\mathrm {th}}$ term of the lower central series $\{\gamma_n(\mathcal V(\mathbb ZG))\}_{n=1}^\infty$.\\

In this section, we collect some results about this filtration.

\begin{description}
	\item[(a)] It is well-known that, for every group $G$, the map $$g\mapsto g-1+\Delta^2(G),\ g\in G,$$ induces an isomorphism  $G/\gamma_2(G)\cong \Delta(G)/\Delta^2(G)$. This map extends to  $\mathcal V(\mathbb ZG)$ to yield \begin{equation}\label{v2}
	\mathcal V(\mathbb ZG)=G\mathcal V_2(\mathbb ZG),\quad 
	\mathcal V(\mathbb ZG )/\mathcal V_2(\mathbb ZG)\cong G/\gamma_2(G)\cong \Delta(G)/\Delta^2(G).\end{equation}
	
	\item[(b)]  In case $G$ is an abelian group, then 
	\begin{equation}\label{T2} 
	\mathcal V(\mathbb ZG)=G\oplus 	\mathcal V_2(\mathbb ZG),
	\end{equation}
	and  $\mathcal V_{2}(\mathbb{Z}G)$ is  torsion-free (see \cite{Seh93}, Theorem 45.1).
	
	\item[(c)]  One of the generic constructions of units in $\mathcal{V}(\mathbb{Z}G)$ is that of bicyclic units. For $g,\,h \in G$, define $$u_{g\,,h}:=1+(g-1)h\hat{g},\quad \mathrm{where}  ~\hat{g}=1+g+\cdots+g^{n-1},$$ $n$ being the order of $g$. The unit $u_{g\,,h}$ is called  a {\it{bicyclic unit}} in $\mathbb{Z}G$; it is trivial if, and only if, $h$ normalizes $\langle g\rangle$, and is of infinite order otherwise.
	 Since $u_{g,\,h}=1+(g-1)(h-1)\hat{g}\equiv 1\mod \Delta^2(G),$ it follows that all  bicyclic units 
	 \begin{equation}\nonumber
	u_{g,\,h}\in \mathcal{V}_{2}(\mathbb{Z}G).
	 \end{equation} Moreover, if $g,\,h$ are of relatively prime  orders, then
	 \begin{equation}\label{eq:bicyclic}
u_{g,\,h}\in \mathcal V_\omega(\mathbb ZG),
	 \end{equation}	  for, in that case, $(g-1)(h-1)\in \Delta^\omega(G):=\cap_{n=1}^\infty(\Delta^n(G)).$ 
\end{description}

Another generic construction of units in $\mathcal{V}(\mathbb{Z}G)$ is that of \emph{ Bass units}. Given $g\in G$ of  order $n$ and  positive integers $k$, $m$ such that $k^m\equiv 1  \mod n$, 
$$u_{k,\,m}(g): = \left(1 +g+\ldots+g^{k-1}\right)^m+\frac{1-k^m}{n}\left(1 +g+\ldots+g^{n-1}\right),$$
is a unit which is trivial, i.e., an element of $G$ if, and only if, $k\equiv \pm 1 ~\mathrm{mod} ~n$.\\

We observe the following elementary but useful fact.

\begin{lemma}\label{TorsionOnly}
		 If $G$ is a finite group, and $ \mathcal V_n(\mathbb ZG)= \{1\}$ for some $n\geq 1$, then all units of $\mathcal V(\mathbb ZG)$ must be torsion; in particular, all Bass and bicyclic units must be trivial.
\end{lemma}
\begin{proof}
	Let $u\in \mathcal V(\mathbb ZG)$.  As $u \in 1 + \Delta(G)$ and $G$ is finite,  therefore, for any $n\geq 1$, there exists $m_n\in \mathbb{N}$ such that $u^{m_n}\in 1 +\Delta^n(G)$, i.e., $u^{m_n}\in \mathcal V_n(\mathbb ZG)=\{1\}$. Hence $u$ is a torsion element.
\end{proof}

A well known result in the theory of units in group rings states that, for a  finite group $G$, all central torsion units in $\mathbb ZG$ must be trivial (see e.g. \cite{Seh93}, Corollary 1.7). Further, a group $G$, such that all central units in $\mathbb ZG$ are trivial, is termed as a {\emph {cut-group}}; this class of groups is currently a topic of active research (\cite{BMP17}, \cite{Mah18}, \cite{Bac18}, \cite{BCJM}, \cite{BMP19}, \cite{Tre19}). In this article, we quite often use the characterization of abelian cut-groups; namely, an abelian group $G$ is a cut-group, if, and only if, its exponent divides $4$ or $6$. \\

We now give a classification of finite groups for which this filtration terminates with identity after a finite number of steps.

\begin{theorem}\label{theo:Finite_Terrminates}
	Let $G$ be a finite group. Then $\mathcal{V}_n(\mathbb{Z}G)=\{1\}$ for some $n\geq 1$ if, and only if, either\begin{description}
		\item[(i)] $G$ is an abelian cut-group; or\item[(ii)] $G=K_8\times E$, where $K_8$ denotes quaternion group of order $8$ and $E$ denotes elementary abelian 2-group.
	\end{description}
\end{theorem}

\begin{proof}
	Let $G$ be a finite group such that $\mathcal{V}_n(\mathbb{Z}G)=\{1\}$, for some $n\geq 1$. Since $\gamma_n(\mathcal{V}(\mathbb{Z}G))\subseteq \mathcal{V}_n(\mathbb{Z}G)$, we have that $\mathcal{V}(\mathbb{Z}G)$ is nilpotent. By classification of finite groups $G$ with nilpotent unit group $\mathcal{V}(\mathbb{Z}G)$ (\cite{Mil76}, Theorem 1), $G$ must be either an abelian group or $G=K_8\times E$. Now, if $G$ is a finite abelian group which is not a cut-group, then there exists a non-trivial Bass unit in $\mathcal{V}(\mathbb{Z}G)$, which in view of Lemma \ref{TorsionOnly}, is not possible.

	Conversely, if $G$ is an abelian cut-group, then $\mathcal{V}_2(\mathbb{Z}G)=\{1\}$ (see (\ref{T2})) and if $G=K_8\times E$, then by Berman-Higman theorem, we have $\mathcal{V}(\mathbb{Z}G)=G$ implying $D_n(G)=\mathcal{V}_n(\mathbb{Z}G),~n\geq 1$. Moreover, in this case $G$ is a nilpotent group of class~$2$, i.e., $\gamma_3(G)=\{1\}$. Since $D_3(G)= \gamma_3(G)$ (see \cite{Pas79}, Theorem 5.10, p.\,66), we obtain  $\mathcal{V}_3(\mathbb{Z}G)=\{1\}$.
\end{proof}

\section{Groups $G$ with $\Delta$-adic residue  of $\mathcal{V}(\mathbb{Z}G)$ trivial}

The triviality of $\Delta$-adic residue of $\mathcal{V}(\mathbb{Z}G)$ naturally restricts the structure of the group $G$, and consequently also of its subgroups, it being a subgroup closed \linebreak property. In this section, we explore the structure of groups with this property.

Before proceeding further, observe that if, for a group $G$, $\Delta^{\omega}(G)=\{0\}$, then trivially $\mathcal{V}_\omega(\mathbb{Z}G)=\{1\}$. A characterization for a group $G$ to have the property $\Delta^{\omega}(G)=\{0\}$ is known; we recall it for the convenience of reader. To this end, we need few definitions, which we give next.

A group $G$ is said to be {\it discriminated by a class $\mathbf{C}$} of groups if, for every finite subset $g_1,\,g_2,\,\ldots\,,\,g_n$ of distinct elements of $G$, there exists a group $H\in \mathbf{C}$ and a homomorphism  $\varphi:G\rightarrow H$, such that $\varphi(g_i)\neq \varphi(g_j)$ for  $i\neq j$. For a class of groups $\mathbf{C}$, a group $G$ is said to be residually in $\mathbf{C}$ if it satisfies:
\begin{quote} For every $1\neq x\in G$, there exists  a normal subgroup $N_x$ of $G$ such that $x\not\in N_{x}$ and $G/N_{x}\in \mathbf{C}$.\end{quote} Clearly, if $G\in \mathbf{C}$, then $G$ is residually in $\mathbf{C}$.
Note that if a class $\mathbf{C}$ is closed under subgroups and finite direct sums, then to say that $G$ is {\it residually in $ \mathbf{C}$} is equivalent to saying that $G$ is discriminated by the  class $\mathbf{C}$. 
\begin{theorem}\label{Delta0}$\mathrm(\cite{Lic77},see~also~ \cite{Pas79},~ Theorem~2.30,~ p.\, 92)$ For a group $G$,
$\Delta^{\omega}(G)=\{0\}$ if, and only if, either $G$ is residually `torsion-free nilpotent', or $G$ is discriminated by the class of nilpotent $p_i$-groups, $i\in I$, of bounded exponents, where $\{p_i,i\in I\}$ is some set of primes.
\end{theorem}

 We now characterise the finite groups $G$ for which $ \mathcal{V}(\mathbb{Z}G)$ has trivial $\Delta$-adic residue. For this, we first prove the following:
 
 \begin{lemma}\label{lem:No_Nilpotent}
 Let $G, H$ be finite groups of relatively prime orders. If $\mathbb{Z}G$ has a non-zero nilpotent element, then $\mathcal{V}(\mathbb{Z}[G\oplus H])$ does not have trivial $\Delta$-adic residue.
 \end{lemma}
\begin{proof}
 Let $G, H$ be finite groups of relatively prime order and let $\alpha(\neq 0)$ be a nilpotent element in $\mathbb{Z}G$, so that $\alpha^k=0$, for some $k\in \mathbb{N}$ and hence $\alpha\in \Delta(G).$ Further, if $h(\neq 1)\in H$, then $h-1 \in \Delta(H)$. By assumption on the orders of $G$ and $H$, we obtain $\alpha(h-1)\in \Delta^{\omega}(G\oplus H)$. Note that as $\alpha$ is nilpotent and as $\alpha$ and $h-1$ commute, $\alpha(h-1)$ is a non-zero nilpotent element in $\mathbb{Z}[G \oplus H]$. Consequently, $1+\alpha(h-1)$ is a non-trivial unit in $\mathcal{V}_\omega(\mathbb{Z}[G \oplus H]).$
\end{proof}

\begin{theorem}\label{Finite_TrivialResidue}
	Let $G$ be a finite group. Then $\mathcal V_\omega(\mathbb{Z}G)=\{1\}$ if, and only if, either 
\begin{description}
	\item[(i)] $G$ is an abelian group of exponent $6$, or;
	\item[(ii)] $G$ is a $p$-group.
\end{description}
\end{theorem}
\begin{proof}
	Let $G$ be a finite group with $\mathcal{V}_\omega(\mathbb{Z}G)=\{1 \}$. If $G$ is not a $p$-group, then let  $z\in G$ be an element of order $pq$, with $p,\ q $  primes, $p<q$. The cyclic subgroup $H:=\langle z\rangle$ of $G$ satisfies $\mathcal{V}_\omega(\mathbb{Z}H)=\{1 \}$. On expressing $z$ as $z=xy$ with elements $x,y$ of orders $p$ and $q$ respectively, we have an exact sequence
	$$1\to \mathcal{V}_\omega(\mathbb{Z}H)\to \mathcal{V}(\mathbb{Z}H)\to \mathcal{V}(\mathbb{Z}\langle x\rangle)\oplus \mathcal{V}(\mathbb{Z}\langle y\rangle)\to 1,$$ induced by the natural projections $H\to \langle x\rangle,\ H\to \langle y\rangle.$ Recall that for a finite abelian group $A$, the torsion-free rank $\rho(\mathcal{V}(\mathbb{Z}A)) $ of the unit group $\mathcal{V}(\mathbb{Z}A)$ is given by the following formula (see \cite{Seh78}, Theorem 3.1, p.\,54):
	$$\rho(\mathcal{V}(\mathbb{Z}A)) =\frac{1}{2}\{|A|+n_2-2c_A+1\},$$	
	where $n_2$ is the number of elements of order 2 in $A$ and $c_A$ is the number of cyclic subgroups in $A$. Thus,
	
	$$\rho(\mathcal{V}(\mathbb{Z}\langle x\rangle ))=
	\begin{cases}
	0,~\mathrm{if}~p=2,\\
	\frac{p-3}{2},~{otherwise}
	\end{cases},~~~\rho(\mathcal{V}(\mathbb{Z}\langle y\rangle ))=\frac{q-3}{2}$$
	and
	$$\rho(\mathcal{V}(\mathbb{Z}H))=\begin{cases}
	q-3,~\mathrm{if}~p=2,\\\frac{pq-7}{2},~\mathrm{otherwise}.
	\end{cases}$$
	Therefore,
	$$\rho(\mathcal{V}(\mathbb{Z}H))>\rho(\mathcal{V}(\mathbb{Z}\langle x\rangle))+\rho(\mathcal{V}(\mathbb{Z}\langle y\rangle)),$$ except possibly when $p=2$ and $q=3$, implying that $G$ must be a (2,3)-group.
	
	Moreover, as $\mathcal{V}(\mathbb{Z}G)$ has trivial $\Delta$-adic residue, it follows that it is residually nilpotent. Hence, by \cite{MW82}, $G$ is a nilpotent group with its commutator subgroup a $p$-group. Thus, if the exponent of $G$ is not $6$, then either $G$ (i)  has an element $x$ of order 4 and an element $y$ of order 3; or (ii) has an element $x$ of order 2 and an element $y$ of order 9. In both the cases, rank considerations, as above yield $\mathcal{V}_\omega(\mathbb{Z}G)\neq\{1\}.$  Consequently, $G=E\oplus S$, where $E$ is an elementary abelian \linebreak $2$-group and $S$ is a group of exponent $3$. As we assume $\mathcal{V}_\omega(\mathbb{Z}G)=\{1\}$, it follows from Lemma \ref{lem:No_Nilpotent} that neither $\mathbb{Z}E$ nor $\mathbb{Z}S$ can have non-zero nilpotent element. By the classification of finite groups whose integral group rings do not have non-zero nilpotent elements \cite{Seh75}, $S$ must be abelian, i.e., $G$ is an abelian group of exponent $6$. 
	
	Conversely, if $G$ is an abelian group of exponent $6$, then $\mathcal{V}(\mathbb{Z}G)=G$ and therefore $\mathcal V_\omega (\mathbb{Z}G)=\{1\}$. Also, if $G$ is a $p$-group, then by Theorem \ref{Delta0}, $  \Delta^\omega(G)=\{0\}$, and hence $\mathcal V_\omega (\mathbb{Z}G)=\{1\}.$ 
\end{proof}

As already observed in (\ref{eq:bicyclic}), the existence of non-trivial bicyclic units\linebreak  $u_{g,\,h}\in \mathbb{Z}G$, with the elements  $g,\,h\in G$ of relatively prime orders, implies that the\linebreak   $\Delta$-adic residue of $\mathcal{V}(\mathbb{Z}G)$ is non-trivial. Thus,  if $G$ is a group with $\mathcal{V}_\omega(\mathbb{Z}G)=\{1\}$, then either $G$ does not have elements of  relatively prime orders, or every \linebreak bicyclic unit $u_{g,\,h}$, with the elements $g,\, h$ of relatively prime orders, is trivial. 
\\

It may be noted  that the triviality of the $\Delta$-adic residue of  $\mathcal{V}(\mathbb{Z}G)$, for an \linebreak  arbitrary (not necessarily finite) group $G$, has an impact on its   torsion elements. 
Following the arguments of Theorem \ref{Finite_TrivialResidue}, we  record the following result which   brings out  constraints on torsion elements of groups with the property under consideration.

\begin{theorem}\label{2,3_only}
	If $G$ is a group with $\Delta$-adic residue of $\mathcal{V}(\mathbb{Z}G)$ trivial, then $G$ cannot have an element of order $pq$ with primes $p<q$, except possibly when $(p,\,q)= (2,\,3)$; 
	in  particular, if the group $G$ is either $2$-torsion-free  or $3$-torsion-free,  then every torsion element of  $G$ has prime-power order.\end{theorem}

Recall that an element $g$ of a group $G$  is said to have {\it infinite $p$-height} in $G$ if, for every choice of natural numbers  $i$ and $j$, there exist elements $x\in G$ and $y\in \gamma_j(G)$ such that $x^{p^i}=gy.$  The set of elements of infinite $p$-height in $G$ forms a normal sugroup; we denote it by $G(p)$.

\begin{theorem}\label{nil}
	Let $G$ be a nilpotent group with $\mathcal{V}_\omega(\mathbb{Z}G)=\{1\}$, and  let $T$ be its torsion subgroup. Then one of the following staements  holds:
	\begin{description}
		\item[(i)] $T=\{1 \}; $
		\item [(ii)] $T$ is a $(2,3)$ group of exponent 6; 
		\item[(iii)] $T$ is a $p$-group,  $T(p)\neq T$, and $T(p)$ is an abelian $p$-group of exponent at most $4$.
	\end{description}

In particular, if $G$ is a nilpotent group with its torsion subgroup  $\{2,\ 3\}$-torsion-free, then, $\mathcal{V}(\mathbb{Z}G)$ has trivial $\Delta$-adic residue only if either $G$ is a torsion-free group or its  torsion subgroup is a $p$-group which has no element of infinite $p$-height.
\end{theorem}
\begin{proof} 	Let $G$ be a nilpotent group with $\mathcal{V}_\omega(\mathbb{Z}G)=\{1\}$, and  let $T$ be its torsion subgroup. If $G$ is not torsion-free, then, in view of Theorem  \ref{2,3_only}, either $T$ is a $p$-group, or must be a $(2,\,3)$ group.  Moreover, in the latter case, argunig as in Theorem \ref{Finite_TrivialResidue}, it follows  that $T$ must be a $(2,\,3)$ group of exponent $6$.

	Suppose $T\ (\not=\{1\})$ is a $p$-group, $T(p) \neq \{1 \}$ and let $x,\,y\in T(p)$. Then  (\cite{Pas79}, Theorem 2.3, p.\,97)  $$(x-1)(y-1)\in \Delta^\omega(T).$$
	Consequently, $${\mathcal{V}_2(\mathbb{Z}[T(p)])\subseteq \mathcal{V}_\omega(\mathbb ZT).}$$ Since $\mathcal{V}_\omega(\mathbb ZT)=\{1	\}$, we have  $$\mathcal{V}_2(\mathbb Z[T(p)])=\{1 \}.$$ Therefore,  by (\ref{v2}), $$\mathcal{V}(\mathbb Z[T(p)])=T(p).$$  Moreover, by (\cite{Mal49}, Theorem 1),   $T(p)$ is a central  subgroup of $T$.  Thus  $T(p)$ is a $p$-group which is an abelian cut-group and thus, its exponent is $2$, $3$ or $4$ and  $T\neq T(p)$.
\end{proof}

\begin{remark}Towards converse of the above result, it may be mentioned that 
	if $G$ is a nilpotent group which is either (i) torsion-free, or (ii) finitely generated having no $p'$-element for some prime $p$, or (iii) $G(p)=\{1\}$, then $\Delta^\omega(G)=\{0\}$, and hence $\mathcal V_\omega(\mathbb ZG)=\{1\}$.

\end{remark}
 	\begin{remark}
	The triviality of the $\Delta$-adic residue $\mathcal V_\omega(\mathbb ZG)$ does not, in general, imply the triviality of $\Delta^\omega(G)$. For instance,  observe that\\ \centerline{$\mathcal{V}(\mathbb{Z}[C_2\oplus {Z}])= C_2\oplus {Z}$ ({\cite{Seh78}, p.\,57}),}\\
and therefore,  $$\mathcal{V}(\mathbb{Z}[C_2\oplus Q])= C_2\oplus  Q\simeq \mathcal{V}(\mathbb{Z}C_2)\oplus \mathcal{V}(\mathbb{Z} Q),$$ where $Z$ is an infinite cyclic group and $Q$ is the additive group of rationals.\linebreak Consequently, $$\mathcal{V}_\omega(\mathbb{Z}[C_2\oplus  Q])=\{1\}.$$ On the other hand, it is easy to see that  $\Delta^\omega(C_2\oplus  Q)\not=\{0\}$. \end{remark}

We next proceed to analyse the case of abelian groups. 

\begin{theorem}\label{abelian}Let $G$ be an abelian group and let $T$ be its torsion subgroup. Then, $\mathcal V_\omega(\mathbb ZG)=\{1\}$ if, and only if, $\mathcal V_\omega(\mathbb ZT)=\{1\}$.  \end{theorem}
\begin{proof}
	Let $G$ be an abelian group and let $T$ be its torsion subgroup. Clearly, if $\mathcal V_\omega(\mathbb ZG)=\{1\}$, then $\mathcal V_\omega(\mathbb ZT)=\{1\}$. 
	
	For the converse, let $u\in \mathcal{V}_\omega(\mathbb ZG)$.
	Then, as $\mathcal V(\mathbb ZG)=\mathcal V(\mathbb ZT)G$ (\cite{Seh78}, p.\,56), we have that $u=vg$, with $v\in  \mathcal V(\mathbb ZT),\ g\in G$. Since  $\mathcal{V}_\omega(\mathbb Z[G/T])=\{1\}$, $G/T$ being torsion-free, projecting $u$ onto $G/T$ yields $g\in T$, and hence $u\in \mathcal V(\mathbb ZT).$ The result now follows from the fact that, $$\mathcal V(\mathbb ZT)\cap \mathcal V_\omega(\mathbb ZG)=\mathcal V_\omega(\mathbb ZT).$$ This is because $$\mathcal V(\mathbb ZT)\cap \mathcal V_n(\mathbb ZG)=\mathcal V_n(\mathbb ZT),\ \text{for all}\ n\geq 1.$$ Note that for the last assertion, it is enough to prove for finitely generated abelian groups, and there it follows from the fact that the group splits over its torsion subgroup.
	\end{proof}

 Theorems \ref{nil} and \ref{abelian} yield  the following: 

\begin{cor}\label{Abelian}
	For an abelian group $G$ which is $\{2,\,3\}$-torsion-free, $\mathcal V(\mathbb ZG)$ has trivial $\Delta$-adic residue  if, and only if, it is either  torsion-free or its torsion subgroup is a $p$-group which has no element of infinite $p$-height.
\end{cor}

\begin{remark}
	If $G$ is a nilpotent group and $T$ is the torsion subgroup of $G$ such that idempotents in $\mathbb QT$ are central in $\mathbb QG$, then $$\mathcal V(\mathbb ZG)=\mathcal V(\mathbb ZT)G. $$The above requirement on idempotents holds true, for instance, if $\mathbb QG$ has no non-zero nilpotent elements (\cite{Seh78}, p.\,194). 
\end{remark}
\begin{prop}
Let $G$ be a nilpotent group and  $T$  its torsion subgroup. If $G/T$ is finitely generated, then $$
\mathcal V_\omega(\mathbb ZG)\cap \mathbb ZT=\mathcal V_\omega(\mathbb ZT).
$$

\end{prop}

\begin{proof}
	Let $G$ be a nilpotent group of class $c$, $Z$ an infinite cyclic group, and   $$1\to H\to G\to Z\to 1$$  a split exact sequence, so that  $$G=HZ,\  H\lhd G,\ H\cap Z=\{1\}.$$ Regard the integral group ring $\mathbb ZH$ as a left $\mathbb ZG$-module with $Z$ acting on $H$ by conjugation and $H$ acting on $\mathbb ZH$ by left multiplication. Then, we have by Swan's Lemma (see \cite{Pas79}, Theorem 2.3, p.\,79),
	\begin{equation}\nonumber\label{module}
	 ^{\Delta^{m^{c}}(G)}\mathbb ZH\subseteq \Delta^m(H),\ \text{for all $m\geq 1$}.
\end{equation}
	Now if $u=n_1g_1+n_2g_2+\ldots n_rg_r\in \mathbb ZG$ is an element of augmentation $1$, and $g_i=h_iz_i$, with $h_i\in H,\ z_i\in Z$, then we have $$^{u-1}1=\sum n_i(h_i-1)\in \Delta(H).$$ Consequently, if $u\in \mathcal V_\omega(\mathbb ZG)$, then$$^{u-1}1\in \Delta^\omega(H).$$
	 It thus follows that we have 
	\begin{equation}\nonumber
	\mathcal V_\omega(\mathbb ZG)\cap \mathbb ZH=\mathcal V_\omega(\mathbb ZH).
	\end{equation}
	
	Next, let $T$ be torsion subgroup of $G$. If $G/T$ is finitely generated, then we have a series $$T\lhd H_1\lhd H_2\lhd \ldots\lhd H_n=G$$ with $H_{i+1}/H_i$ infinite cyclic, $1\leq i\leq n-1$. Induction yields the  desired result, i.e., \begin{equation}\nonumber
	\mathcal V_\omega(\mathbb ZG)\cap \mathbb ZT=\mathcal V_\omega(\mathbb ZT).
	\end{equation}
\end{proof}	

We next shift our analysis to the quotients of a residually nilpotent group by elements of lower central series.

\begin{theorem}\label{nilpotentquotient}
	Let $G$ be a residually nilpotent group, and  let $\{\gamma_{n}(G)\}_{n\geq 1}$ be its lower central series. If 
	$\mathcal V_\omega(\mathbb Z[G/\gamma_n(G)])=\{1\}$ for all $n\geq 1$, then  $\mathcal V_\omega(\mathbb ZG)=\{1\}.$
\end{theorem}

\begin{proof} Let $\pi_n:G\to G/\gamma_n(G)$, $n\geq 1$,  be the natural projection.  Extend $\pi_n$ to the group ring $\mathbb ZG$ by linearity and let the extended map still be denoted by $\pi_n$. Let $u=\sum n_ig_i$ be an element of 
$\mathcal V_\omega(\mathbb ZG)$ with $n_i\in \mathbb Z$ and all $g_i\in G$ distinct. By hypothesis, $\pi_n(u)=\overline{1}$ for all $n\geq 1$.
Choose $m\geq 1$ such that $g_i^{-1}g_j\not\in \gamma_m(G)$ for $i\not= j$. Since $\pi_m(u)=\bar{1}$, it follows that  $u=g\in \gamma_m(G)$. From $\overline{1}=\pi_n(u)=\overline{g}$ for all $n\geq 1$ it follows that $g\in \gamma_n(G)$ for all $n\geq 1$. Since $G$ is residually nilpotent, $g=1.$ Hence $\mathcal V_\omega(\mathbb ZG)=\{1\}$.
\end{proof}

We next prove that the class of groups $G$ with $\mathcal{V}_\omega(\mathbb{Z}G)=\{1\}$ is closed with respect to  discrimination. 
\\

Let $\mathcal C$ denote the class of groups  $G$ such that $\mathcal{V}(\mathbb{Z}G)$ has trivial $\Delta$-adic residue, and let Disc $\mathcal C$ denote  the class of groups discriminated by the class $\mathcal{C}$.

\begin{theorem}\label{disc}
	\quad $\mathcal C=\operatorname{Disc}\,\mathcal C.$
\end{theorem}
\begin{proof} Clearly, by definition, $\mathcal C\subseteq\operatorname{Disc}\,\mathcal C$. 

Let $G\in \operatorname{Disc}\,\mathcal C$, and let $u\in \mathcal V_\omega(\mathbb ZG)$,   $u=\alpha_1g_1+\ldots+\alpha_ng_n$, $g_i\neq g_j$ for $i\neq j$. Let $\varphi:G\to H$ be a homomorphism with $H\in \mathcal C$ such that $\varphi(g_1),\ldots,\varphi(g_n)$ are distinct elements of $H$. Extend $\varphi$ to $\mathbb ZG$ by linearity. Then $\varphi(u)\in \mathcal V_\omega(\mathbb ZH)=\{1\}$, and it thus follows that $u=g\in \ker\varphi$. In case $g\not=1$, we have a homomorphism $\psi:G\to K$ with $K\in \mathcal C$ such that $\psi(g)\not=1=\psi(1).$ Since $g-1\in \Delta^\omega(G)$, extension of $\psi$ to $\mathbb ZG $ shows that $\psi(g-1)=0$, a contradiction. Hence, it follows that $g=1$ and so  $G\in \mathcal C$. 
\end{proof}
\section{Groups with trivial $\mathcal D$-residue}
We next study  groups $G$ whose $\mathcal D$-residue, namely,  $$D_\omega(G):= \cap _{n=1}^\infty D_n(G)$$ is trivial. 
Since $$\gamma_{n}(G)\subseteq D_n(G)\subseteq \mathcal{V}_n(\mathbb{Z}G),$$ the triviality of the $\Delta$-adic residue of a group  always implies that of its $\mathcal D$-residue.\\

 Over the field $\mathbb Q$ of rationals,  the following statements for an arbitrary group $G$ are equivalent (\cite{Pas79}, Theorem 2.26, p.\,90):
 
 \begin{description}
 	\item[(i)] $G$ is residually torsion-free nilpotent;
 	\item[(ii)] $\Delta^\omega_\mathbb Q(G)=\{0\};$
 	\item[(iii)] $D_{\omega,\,\mathbb Q}(G)=\{1\}$. 
 \end{description}
Here $\Delta_\mathbb Q( G) $ denotes the augmentation ideal of the group algebra $\mathbb QG$, $$D_{\omega,\,\mathbb Q}(G)=\cap_{n=1}^\infty D_{n,\,\mathbb Q}(G),\ D_{n,\,\mathbb Q}(G)=G\cap(1+\Delta_{\mathbb Q}^n(G)).$$

\begin{theorem}\label{dimension}
	Let $G$ be a  group such that \begin{description}
		\item[(i)] $D_{\omega,\,\mathbb{Q}}(G)\neq \{1\}$; and
		\item[(ii)] $D_{\omega}(G)=\{1\}$.
	\end{description}
	Then, for any finitely many distinct elements $g_1,\,g_2,\,...\,,\,g_n\in G$, and the subgroup $W$ generated by the left-normed commutators $[w_1,\,w_2,\,\ldots\,w_r], \ r\geq 1,$ with $$w_i\in \{g_j^{-1}g_k,\  1\leq j,\,k\leq n,\ j\neq k\}, \ 1\leq i\leq r,$$ and such that for every $1\leq j,\,k\leq n,\ j\neq k,$ either $g_j^{-1}g_k$ or $g_k^{-1}g_j$ equals some  $w_i$,  one of the following holds:
	\begin{description}
		\item[(a)]
		The subgroup $W$ 
		is contained in $ C_G(D_{\omega,\,\mathbb{Q}}(G))$, the centralizer of $D_{\omega,\, \mathbb{Q}}(G)$ in $G$;  
		\item[(b)] $g_1,\,g_2,\,...\,,\,g_n$ are discriminated by the class $$\mathcal{K}:=\cup_{p\ prime}\,\mathcal{K}_{p},$$ where $\mathcal{K}_{p}$ denotes the class of nilpotent $p$-groups of bounded exponent.		\end{description}
	
\end{theorem}

In order to prove the above result, we need the following:

\begin{lemma}\label{DiscriminatedOrCentralize}
	Let $G$ be a group which is not residually ``torsion-free nilpotent" and let $1\not= g\in D_{\omega,\, \mathbb{Q}}(G)$. If $1\not=g_1\in G$ is such that $1$ and  $g_1$ are not discriminated by the class $\mathcal{K}$, then $[g_1,\,g]\in D_\omega(G).$
\end{lemma}
\begin{proof} Let $G$, $g_1$ and $g$ be as in the statement of the Lemma. Since $$[g_1,\,g]-1=g_1^{-1}g_2^{-1}[(g_1-1)(g-1)-(g-1)(g_1-1)], $$ it suffices to prove that, for all $n\geq 1$, $(g_1-1)(g-1)$ and $(g-1)(g_1-1) $ belong to $ \Delta^n(G)$.
		
	Let $ n\geq 1$ be fixed. Note  that $1\not= g\in D_{n,\,\mathbb{Q}}(G)$ implies that   $$m_n(g-1)\in \Delta^n(G),$$ for some $m_n\in \mathbb{N}$. Let $m_n=p_1^{\alpha_1}p_2^{\alpha_2}...p_r^{\alpha_r}$ be the prime factorization of $m_n$. 	
	Since $1$ and $g_1$ are  not discriminated by the class $\mathcal{K}$, for any group $H\in \mathcal{K}$, and for any homomorphism $\phi:G\rightarrow H$, $\phi(g_1)=1_H$. In particular, for all natural numbers $l$, $k$ and  primes $p$, since $G/\gamma_{l}(G)G^{p^k}\in {\mathcal{K}_p}$, we have that $$g_1\in \cap_{l,k}\gamma_{l}(G)G^{p^k}.$$ Consequently, $g_1-1\in \Delta^n(G)+p^\ell\Delta(G)$ for all primes $p$ and natural numbers $\ell \geq 1$. It then easily follows that both $(g-1)(g_1-1)$ and $(g_1-1)(g-1)$ lie in $\Delta^n(G)$.
\end{proof}

\par\noindent
\textit{Proof of Theorem} \ref{dimension}.
	We first consider the case when the given elements of $G$ are $1$ and $g_1\not=1$, so that $W=\langle g_1\rangle$. 
	
	If $g_1\not\in C_G(D_{\omega,\,\mathbb{Q}}(G))$, then there exists $ 1\not= g\in D_{\omega,\,\mathbb{Q}}(G)$ such that $[g_1,\,g]\neq 1$. But then,  if 1 and $g_1$ are  not discriminated by the class $\mathcal{K}$,  by Lemma \ref{DiscriminatedOrCentralize}, $1\not=[g_1,\,g]\in D_{\omega}(G)$, which contradicts the assumption (ii). Hence,\,1 and $g_1$ must be discriminated by the class $\mathcal{K}$. 
	
	Next, let $g_1,\,g_2,\,\ldots\,,\,g_n$, $n\ge 2$,   be distinct elements of $G$. In case the\linebreak  subgroup $W$ is not contained in $C_G(D_{\omega,\,\mathbb Q}(G))$, then there exists one of the\linebreak  generating  left-normed commutators of $W$, say $u=[w_1,\,\ldots\,,\,w_r]$,   which does not belong to $C_G(D_{\omega,\,\mathbb Q}(G))$. By the case considered above the  elements $1$ and $u$ are  discriminated by the class $\mathcal K$. Thus there exists a  homomorphism $\varphi:G\to H$ with $H\in \mathcal K$ such that $\varphi(u)\not=1$. Since $$\varphi(u)=[\varphi(w_1),\,\ldots\,,\,\varphi(w_r)],$$ it follows that $\varphi(w_i)\not=1$ for $i=1,\ldots,r$. Consequently, $\varphi(g_j^{-1}g_k)\not=1$ for $j\not=k$, since either $g_j^{-1}g_k$ or $g_k^{-1}g_j$ equals one of the $w_i's$. Hence the elements $g_1,\,\ldots\,,\,g_n$ are discriminated by the class $\mathcal K$. $\Box$
\par\vspace{.5cm}
Since, for every torsion group $G$, we have  $D_{\omega,\,\mathbb{Q}}(G)=G$, the above theorem yields the following:
\begin{cor}
	Let $G$ be a torsion group with trivial $\mathcal{D}$-residue, then, every non-central element is discriminated from the identity element  by the class  $\mathcal{K}$.
\end{cor}

For any group $G$,  an analogous  analysis done on $\mathcal V(\mathbb ZG)$   yields the following results of which we omit the proofs. 
\begin{theorem}
	Let $G$ be a group such that  \begin{description}
		\item[(i)] $\mathcal{V}_{\omega,\,\mathbb{Q}}(\mathbb{Z}G)\neq\{1\}$, where $\mathcal{V}_{\omega,\,\mathbb{Q}}(\mathbb{Z}G):=\mathcal{V}(\mathbb{Z}G)\cap(1+\Delta^\omega_\mathbb{Q}(G))$; and
		\item[(ii)] $\mathcal{V}_\omega(\mathbb{Z}G)=\{1\}$
	\end{description}
Then, for finitely many distinct elements $v_1,\,v_2,\,...\,,\, v_n\in \mathcal{V}(\mathbb{Z}G)$, and the subgroup $W$ generated by the left-normed commutators $[w_1,\,\ldots\,,\,w_r]$, $r\geq 1$, with $$w_i\in \{v_{j}^{-1}v_k\,|\,1\leq j,\,k\leq n,\ j\not=k\},\ 1\leq i\leq r,$$ and such that for every $1\leq j,\,k\leq n,\ j\not=k$, either $v_{j}^{-1}v_k$ or $v_k^{-1}v_j$ equals some $w_i$, one of the following holds:
\begin{description}
	\item[(i)] the subgroup $W$ is contained in $ C_{\mathcal{V}}(\mathcal{V}_{\omega,\mathbb{Q}}(\mathbb{Z}G))$, the centralizer of $\mathcal{V}_{\omega,\,\mathbb{Q}}(\mathbb{Z}G)$ in $\mathcal{V}:=\mathcal{V}(\mathbb{Z}G)$; or
	\item[(ii)] there exists a prime $p$ and integers $k,\,\ell$ such that none of the elements $v_1,\,v_2,\,...\,,\, v_n$ belong to $\mathcal{V}_{p,\,k,\,\ell}(\mathbb{Z}G):=\{v\in \mathcal{V}(\mathbb{Z}G):v-1 \in \Delta^k(G)+p^\ell\mathbb{Z}G\}.$
\end{description}
\end{theorem}

\begin{cor}
	If the group $G$ is such that $\mathcal{V}(\mathbb{Z}G)=\mathcal{V}_{\omega,\,\mathbb{Q}}(\mathbb{Z}G)$ and $\mathcal{V}_\omega(\mathbb{Z}G)=\{1\}$, then every non central unit $u\in \mathcal{V}(\mathbb{Z}G)$ does not belong to some $\mathcal{V}_{p,\,k,\,l}(\mathbb{Z}G)$.

\end{cor}
In case $G$ is a finite group, then clearly $\mathcal V_{\omega,\,\mathbb Q}(\mathbb ZG)\not=\{1\}.$ Thus we immediately have the following: 
\begin{cor}
	If  $G$ is a finite cut-group with $\Delta$-adic residue of $\mathcal{V}(\mathbb{Z}G)$ trivial, then  $G$ is nilpotent and every non-central unit is missed by some $\mathcal{V}_{p,\,k,\,l}(\mathbb{Z}G)$.
\end{cor}

\par\vspace{.5cm}
\centerline{\bf Acknowledgement}\par\vspace{.25cm}\noindent
The second author is thankful to Ashoka University, Sonipat, for making available their facilities.

\providecommand{\bysame}{\leavevmode\hbox to3em{\hrulefill}\thinspace}
\providecommand{\MR}{\relax\ifhmode\unskip\space\fi MR }
\providecommand{\MRhref}[2]{%
	\href{http://www.ams.org/mathscinet-getitem?mr=#1}{#2}
}
\providecommand{\href}[2]{#2}

\end{document}